\RequirePackage{fix-cm}
 
\documentclass[smallextended,nospthms,envcountsame]{svjour3}
\smartqed  

\def\makeheadbox{{%
\hbox to0pt{\vbox{\baselineskip=10dd\hrule\hbox
to\hsize{\vrule\kern3pt\vbox{\kern3pt
\hbox{\bfseries Preprint}
\kern3pt}\hfil\kern3pt\vrule}\hrule}%
\hss}}}

\usepackage[T2A,T1]{fontenc}
\usepackage[utf8]{inputenc}
\usepackage{csquotes}
\usepackage[russian,italian,english]{babel}
\usepackage[pdfencoding=unicode,colorlinks=true,citecolor=blue]{hyperref}

\usepackage{amsthm,amsopn,amsmath,amssymb}
\usepackage[nameinlink,capitalize]{cleveref}

\usepackage{microtype}

\usepackage[style=numeric,backend=biber,autolang=other]{biblatex}
\usepackage{booktabs}
\usepackage{mathtools}

\usepackage{tikz}
\usetikzlibrary{positioning,decorations.markings}

\newcommand{\R}{\mathbb{R}}
\newcommand{\N}{\mathbb{N}}
\renewcommand{\d}{\mathrm{d}}
\newcommand{\dx}{\d x}
\newcommand{\dy}{\d y}

\newcommand{\const}{\mathrm{const}}

\newcommand{\Wloc}[1]{W^{1,#1}_{\textup{loc}}(\Omega)}

\DeclareMathOperator{\Int}{\operatorname{int}}
\DeclareMathOperator{\dom}{\operatorname{dom}}
\DeclareMathOperator{\dist}{\operatorname{dist}}

\newtheorem{thm}{Theorem}
\newtheorem{lemma}[thm]{Lemma}
\newtheorem{corollary}[thm]{Corollary}
\newtheorem{prop}[thm]{Proposition}

\DeclarePairedDelimiter\abs{\lvert}{\rvert}

\providecommand\given{\nonscript\;\delimsize|\nonscript\;}
\DeclarePairedDelimiterX\set[1]\{\}{#1}

\DeclarePairedDelimiterX\seq[1](){#1}

\DeclarePairedDelimiter\bracks[]

\newcommand{\mscLink}[1]{\href{https://www.ams.org/mathscinet/msc/msc2010.html?t=#1}{#1}}

\begin{document}

\title{%
	Non-optimality of conical parts for Newton's problem of minimal resistance in the class of convex bodies and the limiting case of infinite height 
}
\titlerunning{Non-optimality of conical parts for Newton's problem and the limiting case}
\date{\today}

\author{%
	Lev Lokutsievskiy
	\and
	Gerd Wachsmuth
	\and
	Mikhail Zelikin
}
\authorrunning{L. Lokutsievskiy, G. Wachsmuth, M. Zelikin}

\institute{
	Lev Lokutsievskiy \at
	Steklov Mathematical Institute of Russian Academy of Sciences, Moscow, Russia,
	ORCID: \href{https://orcid.org/0000-0002-8083-4296}{0000-0002-8083-4296}
	\and
	Gerd Wachsmuth\at
	Brandenburg University of Technology, Cottbus-Senftenberg, Germany,
	ORCID: \href{https://orcid.org/0000-0002-3098-1503}{0000-0002-3098-1503}
	\and
	Mikhail Zelikin \at
	Lomonosov Moscow State University, Moscow, Russia
}

\maketitle
 
\begin{abstract}
	We consider Newton's problem of minimal resistance, in particular
	we address
	the problem arising in the limit
	if the height goes to infinity.
	We establish existence of solutions and
	lack radial symmetry of solutions.
	Moreover,
	we show
	that certain conical parts contained in the boundary
	of a convex body inhibit the optimality
	in the classical Newton's problem with finite height.
	This result is applied to
	certain bodies considered in the literature, which are
	conjectured to be optimal for the classical Newton's problem,
	and we show that they are not.
\end{abstract}
\keywords{
	Newton's problem of minimal resistance \and
	Conical parts \and
	Convex bodies
}
\subclass{
	\mscLink{49K99} \and
	\mscLink{49Q10} \and
	\mscLink{52A15}
}

\section{Introduction}
\label{sec:introduction}
One of the first 
problems in calculus of variations
is a least resistance problem posed by Newton in his \emph{Principia}.
A three-dimensional body with base $\Omega \subset \R^2$ is travelling
in negative $z$-direction.
The upper boundary of the body is given by $\Omega \times \set{0}$,
while the lower boundary is described by the graph of a function $u \colon \Omega \to [-M, 0]$,
where $M > 0$ is the height of the body.
The medium around the body is assumed to be very rare
and under the assumption that each particle collides only once with the body,
one arrives at the resistance
\begin{equation*}
	J(u)
	=
	\int_\Omega \frac{1}{\abs{\nabla u}^2+1} \, \dx\wedge\dy,
\end{equation*}
see \cite{Buttazzo1995,Buttazzo2009}.
In order to comply with the single-impact condition,
one typically considers the convex situation,
namely, $\Omega \subset \R^2$ is assumed to be convex
and $u \colon \Omega \to [-M, 0]$ is convex as well.
We denote the set of all such functions by $C_M\subset\Wloc{2}$.

As we have mentioned, Newton obtained his resistance functional $J$ 
under the assumption of a rare medium. Despite this fact,
in the 20th century, it has been discovered (see \cite[Chapter~III]{HayesProbstein1964}, \cite[\S23]{Cherny}) that
$J$ also describes accurately enough the resistance
of a convex body moving in
dense media with hypersonic speed.
Alternatively, the resistance for hypersonic speeds 
can be computed by the Buseman formula, 
which usually gives better accuracy for non-convex bodies,
but is worse for convex ones \cite[\S23]{Cherny}.

For $\Omega$ being the unit disc,
Newton found an optimal solution among all 
convex bodies of revolution.
Newton's solution has a very non-trivial peculiarity:
its lateral boundary is strictly convex,
but the lower part is a flat disc,
and these parts adjoin each other
by a corner of $45^\circ$.
All standard facts about the problem
can be found in a very well written
survey~\cite{Buttazzo2009}.

Newton's result \cite{Newton} was published in \citeyear{Newton},
exactly $\frac13$ of a millennium ago.
Since that until the end of the 20th century,
it was assumed that
the Newton's body has minimal resistance
among all convex bodies.
Only in 1996, 
Guasoni (in his ``Tesi di Laurea'' \cite{Guasoni}
under the supervision of Buttazzo)
found a ``screwdriver'' shape
that has less resistance
than the one found by Newton
of the same base and height $M\ge 2$.
An analytical argument for the non-optimality of Newton's
solution is given in \cite{BrockFeroneKawohl1996}.

According to \cite[Theorem 2.1]{Buttazzo1995}, 
an optimal body exists 
in the class of convex bodies 
with given base and height. 
There are some analytical results 
on the structure of optimal bodies.
Let $\Omega$ be the unit disc and let the
convex function $u:\Omega\to[-M,0]$ 
describe the shape of an optimal body
for some given height $M>0$. Then

\begin{itemize}
	\item $|\nabla u(x,y)|\in\{0\}\cup[1,+\infty)$ 
	for a.e.\ $(x,y)\in\Omega$, 
	\cite[Theorem~3.2]{Buttazzo2009};
	
	\item $\lim_{(x,y)\to\partial \Omega}u(x,y)=0$, 
	\cite[Theorem~2]{Plakhov2019ANO};
	
	\item if $\omega$ is an open subset of $\Omega$ and $u\in C^2(\omega)$,
	then $u$ is not strictly convex on $\omega$,
	see
	\cite[Remark~3.4]{BrockFeroneKawohl1996} or
	\cite[Lemma~1]{LachandRobertPeletier2001:2}
	(more general results were obtained in \cite[Theorem~1]{LachandRobertPeletier2001:2} and \cite[Theorem~2]{Plakhov2020ExtremePoints});
	
	\item $u$ is not radially symmetric \cite[Theorem~3.4]{Buttazzo2009}.
\end{itemize}

\noindent
Moreover,
this lack of strict convexity implies that
the Euler-Lagrange equations cannot be used to solve the problem,
cf.\ \cite[Theorem~3.5]{Buttazzo2009}.

There are several numerical results \cite{LachandNumeric,Wachsmuth},
which give very good approximations
of optimal bodies
due to \cite[Theorem~2]{LokutsievskiyZelikin}.

In \cite{LokutsievskiyZelikin2}, 
the hypothesis of rotational symmetry 
was replaced by 
the less restrictive hypotheses of 
(i) mirror symmetry w.r.t.\ a vertical plane and 
(ii) developable structure of the side boundary.
Let us remark that all existing aircraft and ships, 
to say nothing of living creatures, have such symmetry.
We have obtained a remarkable formula 
that describes a curve in the plane of symmetry 
and proved that the convex hull of this curve and $\Omega\times\{0\}$ 
is locally optimal in the considered class of admissible bodies,
see \cite[Theorem~9.1]{LokutsievskiyZelikin2}.

\medskip

\textit{The most astonishing fact concerning Newton's problem is that
the exact shapes of optimal bodies in $C_M$
are still unknown.}

\medskip

There were suggested 
a lot of different shapes
as candidates that were considered
as possible solutions to Newton's problem 
in the class of convex bodies,
see \cite{LachandRobertPeletier2001,Wachsmuth,LokutsievskiyZelikin2}.
Some of these profiles contain conical parts on their boundaries. 
We investigate this situation and prove that optimal bodies
cannot contain conical parts of certain type (see \cref{sec:nonoptimality} and \cref{thm:nonoptimality2}).
We use these results in \cref{sec:non-optimal_suggestions} to prove 
non-optimality of all bodies conjectured in the literature.

We also study what is happening in the limiting case $M \to \infty$
by a rescaling $\hat u = u / M$.
It seems that this auxiliary limiting problem was not studied so far, 
but it is extremely useful for studying Newton's problem for large heights (see \cref{subsec:lokutsievskiy_zelikin_solution}).
Our non-optimality result also extends
to this infinite-height case.
Moreover, we reestablished classical results mentioned above 
for this limiting problem. Precisely, 
we prove that an optimal body in the limiting problem exists 
in the class of convex bodies (see \cref{thm:limit_problem}). 
Let $\Omega$ be the unit disc and let the
convex function $u\colon\Omega\to[-1,0]$ 
describe the shape of an optimal body
for the limiting problem. Then we show that

\begin{itemize}
	\item $\lim_{(x,y)\to\partial \Omega}u(x,y)=0$ 
	(in fact, this immediately follows from \cite[Theorem~2]{Plakhov2019ANO});
	
	\item if $\omega$ is an open subset of $\Omega$ and $u\in C^2$,
	then $u$ is not strictly convex on $\omega$ (see \cref{sec:props});
	
	\item $u$ is not radially symmetric (see \cref{thm:limiting_solution_is_not_radially_symmetric}).
\end{itemize}

\section{Notation and Preliminaries}
\label{sec:notation_preliminaries}
Let $\Omega\subset\R^n$ be a
compact convex domain with nonempty interior, i.e., $\Int\Omega\ne\emptyset$.
For some fixed height $M > 0$,
we define the class of functions\footnote{We consider only closed (i.e., lower semicontinuous) convex function due to the following two reasons. First, $J(u)=J(\mathrm{cl}\,u)$ for any convex function $u$, and hence, $\mathrm{cl}\,u$ is a canonical representative for $u$ in the Sobolev space $W^{1,1}_{\textup{loc}}(\Omega)$. Second, for closed convex functions, the mentioned result by Plakhov \cite{Plakhov2019ANO} can be stated in a very nice way: if $u\in C_M$ is optimal, then $u|_{\partial\Omega}=0$.}
\begin{equation*}
	C_M
	:=
	\set[\big]{
		u \colon \Omega \to [-M, 0]
		\given
		\text{$u$ is convex and closed}
	}
	.
\end{equation*}
Note that each $u \in C_M$ is locally Lipschitz
in $\Int\Omega$ and, therefore, differentiable a.e.
Hence, we can define the objective
$J \colon C_M \to \bar\R$ with $\bar\R := \R \cup \set{\infty}$
via
\begin{equation*}
	J(u)
	:=
	\int_\Omega \frac{1}{\abs{\nabla u}^2+1} \, \dx\wedge\dy
	\qquad
	\forall u \in C_M.
\end{equation*}
Now,
Newton's problem of least resistance is
given by
\begin{equation}
	\label{problem:J}
	J(u) \to \min_{u\in C_M}
	.
\end{equation}
The classical case considered by Newton
uses the two-dimensional unit disc
$\Omega := \set{(x,y) \in \R^2 \given x^2 + y^2 \le 1 }$.
In this case,
the problem is rotationally symmetric.
Under the additional condition that the solution
is rotationally symmetric as well,
Newton was able to solve the problem,
see~\cite{Buttazzo2009}.

Buttazzo, Ferone and Kawohl proved in \cite{Buttazzo1995} that there exists an
optimal solution for any $M > 0$ and any $\Omega$ (as above).
This solution might not be unique.
Indeed, on one hand, in the classical case,
it was shown by \cite{BrockFeroneKawohl1996},
that Newton's rotationally symmetric body
is not a solution in the class $C_M$.
On the other hand, it is the unique solution
among all bodies of revolution.
Hence, any optimal solution in $C_M$
cannot be rotationally symmetric.
Since rotations of any solution 
to the classical problem
are also solutions,
a solution cannot be unique.
Moreover, it is clear that the set of solutions depends on the height $M$.

Let us have a brief look into the existence result of \cite{Buttazzo1995}.
We introduce the space
\begin{equation*}
	\Wloc{p}
	:=
	\set[\big]{
		u \colon \Omega \to \R
		\given
		u \in W^{1,p}(K)
		\text{ for all compact subsets $K$ of $\Int\Omega$}
	}
	,
\end{equation*}
where $p \in [1,\infty]$ is arbitrary,
and we say that
$u_n \to u$ in $\Wloc{p}$ if and only if
$u_n \to u$ in $W^{1,p}(K)$ for all compact subsets $K$ of $\Int\Omega$.
Then,
we have the following result,
see \cite[Theorem~2.1 and Lemma~2.2]{Buttazzo1995}.
\begin{lemma}
	\label{lem:J_and_Wloc}
	For all $M > 0$ and any $p \in [1,\infty)$
	we have $C_M \subset \Wloc{p}$.
	The set $C_M$ is sequentially compact in $\Wloc{p}$,
	i.e., any sequence $\seq{u_n} \subset C_M$
	has a subsequence $\seq{u_{n_k}}$ with $u_{n_k} \to u$ in $\Wloc{p}$
	for some $u \in C_M$.
	Moreover, 
	$u_{n_k} \to u$ everywhere and $\nabla u_{n_k} \to u$ a.e.\ in $\Omega$.
	The 
	functional $J$ is
	sequentially lower semicontinuous on $\Wloc{p}$.
\end{lemma}
With this lemma,
the existence of minimizers of \eqref{problem:J}
follows from the direct method of calculus of variations,
see \cite[Theorem~2.1]{Buttazzo1995}.

\section{The Limiting Case of Infinite Height}
\label{sec:limiting_case}
In this section, we study the limit of optimal solutions as $M\to\infty$.
It is easy to see that the minimum of any optimal solution in $C_M$ is $-M$.
Hence, if we want to find the limit shape, we need to reformulate the problem.
Consider the following problem
\begin{equation}
	\label{problem:J_M}
	J_M(u)
	:=
	\int_\Omega \frac{1}{\abs{\nabla u(x)}^2+M^{-2}} \, \dx
	\to
	\min_{u\in \hat C}
	,
\end{equation}
where $\hat C=C_1$ is the set of convex functions $u$ with $\dom u=\Omega$ and $-1\le u\le 0$.
Note that we use $\hat C$ instead of $C_1$ to avoid confusion
with the continuously differentiable functions $C^1$.

Obviously, $J_M(u) = M^2\,J(M\,u)$
and $u \in \hat C$ if and only if $M \, u \in C_M$.
Thus, if $u_M\in \hat C$ is an
optimal solution to problem~\eqref{problem:J_M} then $Mu_M\in C_M$ is an
optimal solution to problem~\eqref{problem:J} and vice versa.
Solutions $\hat u_M$ are bounded in $\Omega$,
and we are interested in a limit (in some sense)
of these solutions as $M\to\infty$.

Problem~\eqref{problem:J_M} is closely connected with the following problem
with limit functional:
\begin{equation}
\label{problem:J_infty}
	J_\infty(u) = \int_\Omega \frac{1}{\abs{\nabla u(x)}^2} \, \dx \to\min_{u\in \hat C}.
\end{equation}
Again, the existence of minimizers follows from \cref{lem:J_and_Wloc},
see \cite[Theorem~2.1]{Buttazzo1995}.

First, let us show how minima in problems~\eqref{problem:J_M}
and~\eqref{problem:J_infty} are connected.

\begin{thm}
\label{thm:limit_problem}
	Let $p \in [1,\infty)$ be given and $\seq{u_M}_{M > 0}$ denote a family of solutions to problems \eqref{problem:J_M}.
	For every increasing sequence $\seq{M_n}_{n \in \N}$ with $M_n \to \infty$,
	the sequence $\seq{u_{M_n}}_{n \in \N}$ possesses an accumulation point
	in $\Wloc{p}$.
	Every such accumulation point is a solution to \eqref{problem:J_infty}.
	Moreover,
	\[
		\lim_{M\to\infty} J_M(u_M) = 
		\min_{u\in \hat C} J_\infty(u) < \infty.
	\]
\end{thm}
\begin{proof}
	Due to $u_M \in \hat C$ for all $M > 0$,
	\cref{lem:J_and_Wloc} implies the claimed existence of accumulation points.
	Now, for any sequence $\seq{u_{M_n}}$ with $M_n \to \infty$
	and
	$u_{M_n} \to \hat u$ in $\Wloc{p}$, we have
	(along a subsequence)
	$\nabla u_{M_n} \to \nabla \hat u$ a.e.\ in $\Omega$.
	Hence,
	Fatou's lemma implies
	\begin{equation*}
		J_\infty( \hat u )
		=
		\int_\Omega \frac{1}{\abs{\nabla \hat u}^2} \, \dx
		\le
		\liminf_{n \to \infty}
		\int_\Omega \frac{1}{M_n^{-2} + \abs{\nabla u_{M_n}}^2} \, \dx
		=
		\liminf_{n \to \infty} J_{M_n}(u_{M_n})
	\end{equation*}
	On the other hand, we trivially have
	$J_M(u) \le J_\infty(u)$
	for all $M > 0$ and $u \in \hat C$.
	Hence, the optimality of $u_{M_n}$ implies
	\begin{equation*}
		\forall u \in \hat C
		\quad
		J_\infty(\hat u)
		\le
		\liminf_{n \to \infty} J_{M_n}(u_{M_n})
		\le
		\liminf_{n \to \infty} J_{M_n}( u )
		\le
		J_\infty(u).
	\end{equation*}
	This shows that $\hat u$ is a solution to \eqref{problem:J_infty}.

	From $J_M(u)\ge J_{M'}(u)$ for $M\ge M'$,
	we get that $\inf_{u\in \hat C}J_M(u)$ is monotonically increasing in $M$.
	Hence,
	\[
		\lim_{M\to\infty} J_M(u_M)=
		\lim_{M\to\infty}\inf_{u\in \hat C} J_M(u)
		\le
		\inf_{u\in \hat C}J_\infty (u)
		=
		J_\infty(\hat u).
	\]
	
	It remains to prove
	$J_\infty(\hat u)<\infty$.
	Without loss of generality $0 \in \Int \Omega$.
	Consider $u(x)=-1+|x|/R$ where $R=\max_{x\in\Omega}\dist(x,0)$.
	Obviously, $u\in \hat C$
	and $\abs{\nabla u(x)}\equiv 1/R$ (except for $x=0$).
	Hence, $J_\infty(\hat u)\le J_\infty(u)=R^2\,\operatorname{area}(\Omega)<\infty$,
	since $\Omega$ is compact.
\end{proof}

In \cite{LokutsievskiyZelikin2},
an important subclass $E_M\subset C_M$
for the classical case $\Omega=\{x^2+y^2\le 1\}\subset\R^2$ is considered.
The subclass $E_M$ consists of functions being a convex envelope
of $\delta_\Omega$ and a convex curve 
lying in a vertical plane of symmetry
(see Section 4 in \cite{LokutsievskiyZelikin2} for details).
In \cite{LokutsievskiyZelikin2},
a family of functions $\tilde u_M\in E_M$ of special form is constructed.
Moreover, it is analytically proved that $\tilde u_M$
is a local minimum for large enough $M$ in $E_M$
w.r.t.\ a certain class of variations, see \cite[Theorem~9.1]{LokutsievskiyZelikin2}.
It is known that the resistances of analytically found $\tilde u_M\in E_M$
and numerically found optimal solution
$\hat u_M$ in $C_M$ (see \cite{Wachsmuth,LachandNumeric})
coincide up to 1\% for $M=1.5$.
In this paper, we will present a new result on optimality
of certain conical parts of the body side boundary,
which allows us investigate the question
whether $\tilde u_M\in E_M$ are optimal in $C_M$ or not.
On the first glance,
they seems to be not optimal,
since the numerical results are accurate enough
and give a slightly better values of the resistance functional.
But the following question is much more interesting:
does the family $\tilde u_M\in E_M$
is at least asymptotically optimal in $C_M$ for $J$
(see conclusion section in \cite{LokutsievskiyZelikin2}).
This question is equivalent to the following:
does the family $M^{-1}\tilde u_M$
is asymptotically optimal in $\hat C$ for $J_M$.
Recall that a family $(u_M)$
is called asymptotically optimal
for functional $J_M$ as $M\to\infty$ if
\[
	\lim_{M\to\infty}\frac{J_M(u_M)}{\inf_{u\in \hat C}J_M(u)}=1.
\]

The following proposition gives a simple way to work with asymptotically
optimal families,
it can be proved analogously to \cref{thm:limit_problem}.

\begin{prop}
\label{prop:asymp_limit}
	Let $p \in [1,\infty)$ be given and
	consider an asymptotically optimal family $u_M\in \hat C$ for $J_M$ as $M\to\infty$.
	Then there exists a sequence $M_k\to\infty$ as $k\to\infty$
	and $u_\infty\in \hat C$,
	$u_{M_k} \to u_\infty$ in $\Wloc{p}$
	and
	$u_\infty$ is optimal in $\hat C$ for limit functional $J_\infty$.
\end{prop}

This proposition gives us a tool to check if a certain family of bodies is
asymptotically optimal.
Together with results in the next section it allows us to investigate the
family found in \cite{LokutsievskiyZelikin2}.

\section{Properties of Solutions to the Limiting Problem}
\label{sec:props}

Suppose that $u$ is a solution of the following problem
\[
J(u)=\int_{\Omega} f(\nabla u(x))\,\d x\to\min_{u\in \hat C}
\]

We give a simple proof of the following well known fact (see also~\cite[Theorem 1]{Plakhov2020ExtremePoints}).

\begin{prop}
	\label{prop:lack_of_strict_convexity}
	Let $x_0\in\Int\Omega$. Suppose that $f''(\nabla u(x_0))$ has at least 1 negative eigenvalue and $u$ is $C^2$ in a neighborhood of $x_0$.
	Then $\det u''(x_0)=0$. 
\end{prop}

\begin{proof}
	First, let $x_0$ be a maximum of $u$. In this case, $u\equiv\const$, since $u$ is convex and $x_0\in\Int\Omega$. 
	Second, let $x_0$ not be a minimum of $u$, i.e.\ $u(x_0)>-1$. Suppose the contrary: let $\det u''(x_0)>0$. 
	Then $u+h\in\hat C$, if $h\in C^2$, $\|h\|_{C^2}$ is small enough, and
	$\mathrm{supp}\,h$ belongs to a neighborhood of $x_0$ where $\det u''$ is separated from 0.
	Hence $u$ is a local minimum of $J$ under the described variations.
	Thereby,  in the neighborhood, $u$ must satisfy both the Euler-Lagrange equation (which is not important for us) 
	and the Legendre condition, since $u(x)\in\R^1$ (see \cite{Hadamard,Bliss}).
	The last one states that
	the Hessian form $f''(\nabla u(x_0))$ must be non-negative definite and this is a contradiction. 
	Finally, let $x_0$ be a minimum of $u$.
	Again, we assume $\det u''(x_0) > 0$.
	Then, we can apply the second part of the proof in a neighborhood of $x_0$
	and obtain
	$\det u''(x)=0$
	for all $x$ in a punctured neighborhood of $x_0$.
	This contradicts $\det u''(x_0)>0$.
\end{proof}

For the limiting problem~\eqref{problem:J_infty}, we have $f(p)=|p|^{-2}$. 
Eigenvalues of $f''(p)$ are $-2|p|^{-4}$ and $6|p|^{-4}$. 
Therefore strict convexity of $C^2$ parts is forbidden for optimal solutions.

Similar to the classical Newton's problem,
we are able to prove that solutions to the limiting problem~\eqref{problem:J_infty} cannot be radially symmetric.

\begin{thm}
	\label{thm:limiting_solution_is_not_radially_symmetric}
	Let $n=2$ and $\Omega=\{x_1^2+x_2^2\le 1\}$. Then any solution to the limiting problem~\eqref{problem:J_infty} is not radially symmetric.
\end{thm}

\begin{proof}
	We prove that the problem restricted to radial symmetric solutions
	is uniquely solvable
	and
	show that the solution has strictly convex smooth parts that contradicts \cref{prop:lack_of_strict_convexity}. Let $u(x_1,x_2)=z(r)$ where $r=\sqrt{x_1^2+x_2^2}$. Then problem~\eqref{problem:J_infty} becomes
	\begin{equation}
	\label{problem:radial_limit}
	\int_0^1\frac{r\,\d r}{z'^2(r)}\to\min_z,\qquad z(0)=-1,\ z(1)=0,\ z\text{ is convex and monotone}.
	\end{equation}
	This problem is similar to the classical Newton's problem and can be solved similarly. So let us find a solution in the class of monotonic functions, and show that it is convex and gives absolute minimum to problem~\eqref{problem:radial_limit}. So,
	\[
	\int_0^1\frac{r\,\d r}{w^2(r)}\to\min_z,\qquad z(0)=-1,\ z(1)=0,\ z'(r)=w(r)\ge 0.
	\]
	To apply the Pontryagin maximum principle (PMP), we define the Pontryagin function
	\[
	H=-\frac{\lambda_0r}{w^2} + q w,
	\]
	where $q=q(r)$ is conjugate to $z$ and $\lambda_0\in\R$ is non-negative. Hence $q'=-H_z=0$ and $q(r)\equiv q_0=\const$.
	
	Suppose that $\lambda_0=0$, then the optimal $w\ge 0$ maximizes $q_0w$ due to PMP. Hence $q_0<0$ (the case $\lambda_0=q_0=0$ is forbidden by PMP) and $w(r)=0$ for all $r\in[0,1]$, which contradicts to $z(1)-z(0)=1$. 
	
	So we put $\lambda_0=1/2$ and the PMP gives the following finite dimensional problem
	\begin{equation}
	\label{eq:maximum_principle_for_limit_radial_problem}
	H=-\frac{r}{2w^2} + q_0 w \to \max_{w\ge 0}.
	\end{equation}
	This function is concave for $w\in(0,\infty)$ and goes to $-\infty$ as $w\to+0$. If $q_0\ge 0$ there is no maximum. Hence $q_0< 0$ and $H$ goes to $-\infty$ also as $w\to+\infty$. Hence $H$ achieves its global maximum at the point where $H_w=0$,
	i.e.,
	\[
	H_w=\frac{r}{w^3}+q_0=0
	\quad\Rightarrow\quad
	w=-(q_0r)^{1/3}
	.
	\]
	Therefore $z=ar^{4/3}-x_0$ where $a=-\frac34(q_0)^{1/3}$. Since $z(0)=-1$ and $z(1)=0$, we have a (convex and monotone) candidate $\hat z=r^{4/3}-1$.
	
	Let us now prove, that $\hat z=r^{4/3}-1$ is the unique solution to problem~\eqref{problem:radial_limit}. Indeed, let $z(r)$ be an arbitrary convex monotone function with $z(0)=-1$ and $z(1)=0$. Since $\hat w=\hat z'$ is global maximum of $H$, and $z'(r)\ge 0$, then
	\[
	-\frac{r}{2z'^2} + q_0 z'\le -\frac{r}{2\hat z'^2} + q_0 \hat z'.
	\]
	Integrating this inequality for $r\in[0,1]$ we obtain
	\[
	-\frac12 \int_0^1\frac{r\,\d r}{2z'^2} + q_0(z(1)-z(0)) \le -\frac12 \int_0^1\frac{r\,\d r}{2\hat z'^2} + q_0(\hat z(1)-\hat z(0))
	\]
	or
	\[
	\int_0^1\frac{r\,\d r}{2z'^2} \ge \int_0^1\frac{r\,\d r}{2\hat z'^2}.
	\]
	
	So we have proved that $\hat u(x_1,x_2) = (x_1^2+x_2^2)^{2/3}-1$ is
	the unique global minimum of $J_\infty$ in the $\hat C$ subclass of radially symmetric bodies. 
	It remains to compute $f''(p)$ for $f(p)=|p|^{-2}$ due to \cref{prop:lack_of_strict_convexity}.
	It is easy to check, that $f''(p)$ has eigenvalues $-2|p|^{-4}$ and $6|p|^{-4}$. 
	Hence, using \cref{prop:lack_of_strict_convexity} we obtain that the unique global
	minimum $\hat u$ of $J_\infty$ is the $\hat C$ subclass of radially symmetric
	bodies cannot be solution to the limiting problem~\eqref{problem:J_infty}.
\end{proof}

We note that the objective value of the radial solution is given by
\begin{equation*}
2 \pi \int_0^1 \frac{r}{z'(r)^2} \, \d r
=
\frac98 \pi \int_0^1 r^{1/3} \, \d r
=
\frac{27}{32} \pi \approx 2.651.
\end{equation*}
A simple screwdriver-shape given by the convex hull of
$\Omega \times \set{0}$ and the line segment
joining $(\pm a, 0, -1)$
with $a \approx 0.55527$
yields the better value of
approx.~$2.145$.

In the case of finite height,
solutions satisfy
$\abs{\nabla u(x,y)} \in\{0\}\cup[1,+\infty)$ 
for a.e.\ $(x,y)\in\Omega$.
This seems not to be true for the solution of the limiting problem.
Indeed, we observed gradients of magnitude approx.\ $0.9863$
in numerical simulations.
Detailed results of the numerical computations might appear elsewhere.

\section{Non-optimality of Conical parts}
\label{sec:nonoptimality}
In this section, we will prove a non-optimality result for certain conical parts
included in the boundary of the body
in the classical situation of a circular base
$\Omega=\set{(x,y) \in \R^2 \given x^2+y^2\le 1}$.
In other words, we will prove that the boundary of an optimal body cannot have
certain conical parts.

We will write $\delta=M^{-2}$ for short.
Hence, $\delta\ge 0$, and the case $\delta=0$ corresponds to $M=\infty$.
Therefore,
\[
	J_M(u) = \int_\Omega \frac{1}{\abs{\nabla u(x)}^2+\delta} \, \dx,
\]
and the function $u$ is normalized, i.e.,
$-1\le u\le 0$.

We start by considering a simple situation,
in which the entire body is just an oblique circular cone.
The base is given by $\Omega \times \set{0}$
and the apex is given by
the point
$P_0=(x_0,y_0,-1)$ with $(x_0, y_0) \in \Int\Omega$.
We take a different point
$(x_1, y_1) \in \Int\Omega$.
We further take some height $M_1 > 0$, such that
$(x_1, y_1, -M_1)$
lies exactly on the boundary of the cone.
Now, for $\varepsilon > 0$ we consider the perturbed point
$P_1 = (x_1, y_1, -m)$ with $m = M_1 + \varepsilon^2 \, (1-M_1)$.
The perturbed body is given by the convex hull of
the base $\Omega \times \set{0}$
and the points $P_0$ and $P_1$,
see \cref{fig:visulization_perturbation}.
\begin{figure}[ht]
	\centering
	\begin{tikzpicture}[scale=2.0,
			mydot/.style={draw,fill,circle,inner sep=0pt,minimum size=2pt}
		]
		\draw (0,0) circle (1);
		\newcommand\xo{.3}
		\newcommand\yo{.2}
		\newcommand\Mi{.5}
		\newcommand\eps{0.5}
		\renewcommand\xi{(1 - \Mi + \xo * \Mi)}
		\newcommand\yi{(\Mi * \yo)}
		\newcommand\xii{(1-\eps*\eps*\xo)/(1-\eps*\eps)}
		\newcommand\yii{(0-\eps*\eps*\yo)/(1-\eps*\eps)}
		\newcommand\rii{sqrt( (1 - 2*\eps^2 * \xo + \eps^4*(\xo^2 + \yo^2)) / ( ( 1 - \eps^2)^2 ) )};
		\newcommand\thetaii{(-atan( (\eps^2*\yo) / (1-\eps^2*\xo) ))};
		\newcommand\phip{(\thetaii + acos(1/\rii))}
		\newcommand\phim{(\thetaii - acos(1/\rii))}
		\node[mydot](O) at (0,0){};
		\node[left=.1 of O]{$(0,0)$};
		\node[mydot](xo) at ({\xo}, {\yo}){};
		\node[mydot](xi) at ({\xi}, {\yi}){};
		\node[mydot](phip) at ({cos(\phip)}, {sin(\phip)}){};
		\node[mydot](phim) at ({cos(\phim)}, {sin(\phim)}){};
		\draw (xo) -- (phip) -- (xi) -- (xo);
		\draw (xo) -- (phim) -- (xi);

		\node[mydot](xii) at ({\xii}, {\yii}){};
		\draw[gray] (xi) -- (xii) -- (phip);
		\draw[gray] (xii) -- (phim);

		\node[left=.1 of xo]{$(x_0,y_0)$};
		\node[right=.1 of xi]{$(x_1,y_1)$};
		\node[right=.1 of xii]{$(x_2,y_2)$};
		\node[right=.1 of phip]{$\varphi_+$};
		\node[right=.1 of phim]{$\varphi_-$};
	\end{tikzpicture}
	\hfill%
	\includegraphics[width=.5\textwidth]{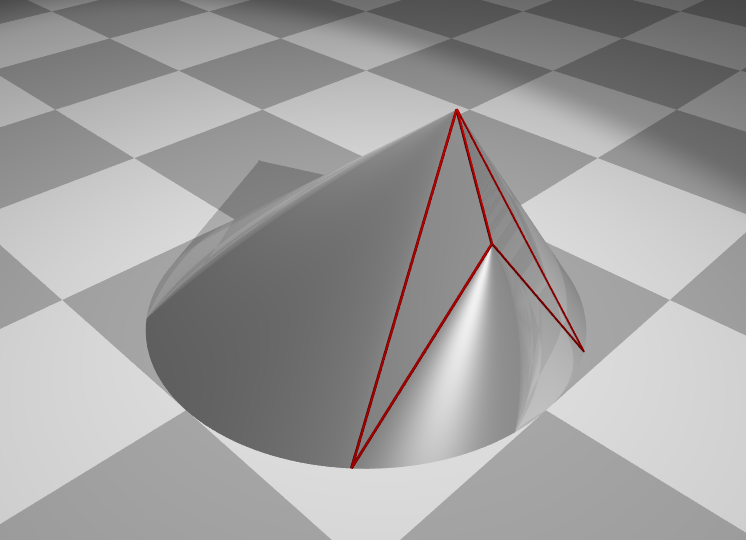}
	\caption{Visualization of the perturbed surface.}
	\label{fig:visulization_perturbation}
\end{figure}
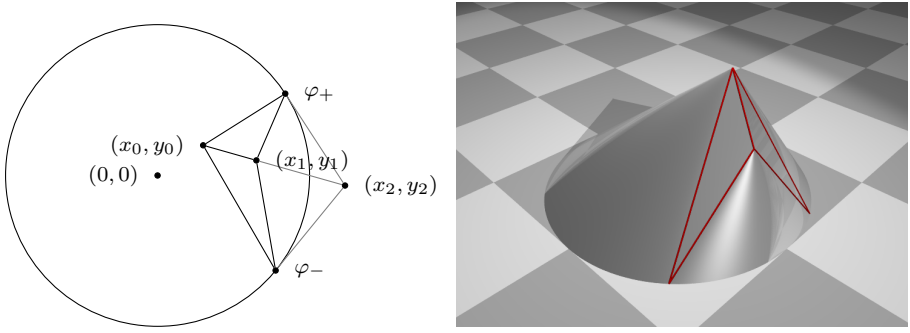
In the following, we derive an expansion formula
of the resistance of the perturbed body
in terms of the parameter $\varepsilon$.
Note that the original cone corresponds to $\varepsilon = 0$.

Since the resistance does not change under rotations and reflections,
we can assume without loss of generality,
that the line through the points $(x_0,y_0)$ and $(x_1,y_1)$
also contains the point $(1,0)$
and that $y_0\ge0$.

The line $P_0P_1$ intersects the horizontal plane $\{z=0\}$ at the point
\[
	\begin{pmatrix}
		x_2\\y_2
	\end{pmatrix}
	=
\frac{m}{m-1}
\begin{pmatrix}
x_0\\y_0\\
\end{pmatrix}
+\frac{1}{1-m}
\begin{pmatrix}
x_1\\y_1\\
\end{pmatrix}
.
\]
The perturbed body can be described by
the four parameters $x_0$, $y_0$, $M_1 \in (0,1)$ and $\varepsilon>0$,
since
the point $P_1$ is given by
\begin{equation*}
	x_1=1-M_1+x_0M_1
	\quad\text{and}\quad
	y_1=M_1y_0.
\end{equation*}
Plugging this into the above equation, we find
\[
	\begin{pmatrix}
		x_2\\y_2\\
	\end{pmatrix}
	=
	\frac{1}{1-\varepsilon^2}
	\begin{pmatrix}
		1-\varepsilon^2 x_0\\
		-\varepsilon^2 y_0\\
	\end{pmatrix}
	.
\]
Next,
we write the point $P_2$ in polar coordinates,
i.e.,
$(x_2, y_2) = r_2 \, (\cos\theta_2, \sin\theta_2)$
with
\[
	r_2^2
	=
	\frac{1-2\varepsilon^2x_0 + \varepsilon^4(x_0^2+y_0^2)}{(1-\varepsilon^2)^2}
	\quad\mbox{and}\quad
	\theta_2=-\arctan\frac{\varepsilon^2 y_0}{1-\varepsilon^2 x_0}
	,
\]
where we used $y_2 \le 0$.
We compute some parameters to describe the structure of the perturbed body.
The circle $\partial\Omega$ contains
two important points $(\cos\varphi_\pm,\sin\varphi_\pm)$,
where $\pm\varphi_\pm>0$ and $\varphi_\pm\to0$ as $\varepsilon\to +0$.
These are the tangent points of the tangent lines to the unit disc passing through $(x_2,y_2)$.
The lateral boundary of the body consists of
the following parts:
\begin{enumerate}
	\item
		A big conic surface with
		apex $(x_0,y_0,-1)$
		and boundary arc $(\cos\varphi,\sin\varphi,0)$
		for $\varphi\in[\varphi_+, \varphi_-+2\pi]$.
		Let us compute the total resistance of this surface.
		We parametrize this part of the boundary
		via
		\begin{equation*}
			\begin{pmatrix}
				x \\ y \\ z
			\end{pmatrix}
			=
			\begin{pmatrix}
				(1-\lambda)x_0+\lambda\cos\varphi
				\\
				(1-\lambda)y_0+\lambda\sin\varphi
				\\
				-(1-\lambda)
			\end{pmatrix}
			,
			\qquad
			\lambda \in [0,1], \varphi \in [\varphi_+, \varphi_- + 2 \pi].
		\end{equation*}
		Note that a normal vector of this surface is given by
		\begin{equation*}
			n =
			\begin{pmatrix}
				\cos\varphi & \sin\varphi & x_0\cos\varphi+y_0\sin\varphi-1
			\end{pmatrix}
			^\top,
		\end{equation*}
		hence we have
		$|\nabla u(x,y)|=1/(1-x_0\cos\varphi-y_0\sin\varphi)$,
		where $(x,y)$ is linked with $(\lambda, \varphi)$
		via the above parametrization.
		For the area of the surface element, we get
		\[
			\d x\wedge \d y
			=
			\lambda\,(1-x_0\cos\varphi-y_0\sin\varphi)\,\d\lambda\wedge \d\varphi.
		\]
		Hence,
		the total resistance is given by
		\begin{align*}
			R_0
			&=
			\int_0^1\lambda\,\d\lambda\left(\int_0^{2\pi} - \int_{\varphi_-}^{\varphi_+}\right)
			\frac{1-x_0\cos\varphi-y_0\sin\varphi}{(1-x_0\cos\varphi-y_0\sin\varphi)^{-2} + \delta} \,\d\varphi \\
			&=
			\frac12\left(\int_0^{2\pi} - \int_{\varphi_-}^{\varphi_+}\right)\frac{(1-x_0\cos\varphi-y_0\sin\varphi)^3}{1+\delta(1-x_0\cos\varphi-y_0\sin\varphi)^2}\,\d\varphi.
		\end{align*}
	
	\item
		A small conic surface consisting of the apex $(x_1,y_1,-m)$
		and the boundary arc $(\cos\varphi,\sin\varphi,0)$
		for $\varphi\in[\varphi_-, \varphi_+]$. Similarly,
		we arrive at
		\[
			R_1
			=
			\frac12\int_{\varphi_-}^{\varphi_+}\frac{(1-x_1\cos\varphi-y_1\sin\varphi)^3}{m^2 + \delta(1-x_1\cos\varphi-y_1\sin\varphi)^2}\,\d\varphi
			.
		\]
	\item
		Two triangles with vertices $(x_0,y_0,-1)$, $(x_1,y_1,-m)$
		and $(\cos\varphi_\pm,\sin\varphi_\pm,0)$.
		On these triangles we have
		$\nabla u_{\pm}=1/(1-x_0\cos\varphi_\pm-y_0\sin\varphi_\pm)$.
		The areas of
		their projections onto the plane $\{z=0\}$ are
	\[
		S_\pm = \pm\frac12
		\det\begin{bmatrix}
				\cos\varphi_\pm-x_0 & \cos\varphi_\pm-x_1 \\
				\sin\varphi_\pm-y_0 & \sin\varphi_\pm-y_1 
		\end{bmatrix}
		.
	\]
\end{enumerate}
Hence, the total resistance of the perturbed body
is given by
the expression
\[
	\mathfrak{R}(\varepsilon) = R_0 + R_1 + \frac{S_+}{\abs{\nabla u_+}^2 + \delta} + \frac{S_-}{\abs{\nabla u_-}^2 + \delta}
	.
\]
In what follows,
we will derive
an asymptotic expansion of $R$
as $\varepsilon \searrow 0$.
Note that
the resistance of the unperturbed
body is given by
\begin{equation*}
	\mathfrak{R}(0)
	=
	R_0(0)
	=
	\frac12 \int_0^{2\pi} \frac{(1-x_0\cos\varphi-y_0\sin\varphi)^3}{1+\delta(1-x_0\cos\varphi-y_0\sin\varphi)^2}\,\d\varphi.
\end{equation*}

It is easy to see that
\begin{equation*}
	\varphi_\pm
	= \theta_2 \pm \arccos\frac{1}{r_2}
	=
	-\arctan\frac{\varepsilon^2 y_0}{1-\varepsilon^2 x_0} \pm \arccos\frac{1-\varepsilon^2}{\sqrt{1-2\varepsilon^2x_0 + \varepsilon^4(x_0^2+y_0^2)}}
	.
\end{equation*}
This right-hand side cannot be used to obtain
an expansion of $\varphi_\pm$,
since the argument of $\arccos$ goes to 1
as $\varepsilon \searrow 0$.
Nonetheless,
by using the addition theorems for cosine and sine,
a straightforward computation gives
\[
	\cos\varphi_\pm
	=
	\frac{(1-x_0\varepsilon^2)(1-\varepsilon^2) \pm y_0\varepsilon^3\sqrt{2(1-x_0)-(1-x_0^2-y_0^2)\varepsilon^2}}{1-2\varepsilon^2x_0 + \varepsilon^4(x_0^2+y_0^2)},
\]
\[
	\sin\varphi_\pm
	=
	\frac{\pm(1-x_0\varepsilon^2)\varepsilon\sqrt{2(1-x_0)-(1-x_0^2-y_0^2)\varepsilon^2} - y_0(1-\varepsilon^2)\varepsilon^2}{1-2\varepsilon^2x_0 + \varepsilon^4(x_0^2+y_0^2)}.
\]
\noindent
Using the last formula, we see that $\pm\varphi_\pm = \varepsilon \, \sqrt{2 \, (1 - x_0)} + O(\varepsilon^2)$ as $\varepsilon\to +0$.
Note that $\varphi_\pm$ were initially defined for $\varepsilon\ge 0$.
However,
they are analytic functions of $\varepsilon\ge 0$ as $\arcsin$ is analytic.
Hence, we are able to extend their domains for $\varepsilon<0$ by analyticity.
Both $\varphi_+$ and $\varphi_-$ become analytic functions of $\varepsilon$
around $0$.
Moreover, $\varphi_+(-\varepsilon)=\varphi_-(\varepsilon)$

First, let us compute expansions for the integrals
appearing in $R_0$ and $R_1$.
Using the Leibniz integral rule,
we arrive at
\begin{equation}
	\frac12\int_{\varphi_-(\varepsilon)}^{\varphi_+(\varepsilon)}\frac{(1-x_0\cos\varphi-y_0\sin\varphi)^3}{1+\delta(1-x_0\cos\varphi-y_0\sin\varphi)^2}\,\d\varphi
	=
	\frac{\sqrt{2} (1-x_0)^{7/2}}{1+\delta(1-x_0)^2} \varepsilon + O(\varepsilon^2)
	.
\label{eq:original_resistance_series}
\end{equation}
Similarly, using $m=M_1+\varepsilon^2(1-M_1)$,
the following expansion can be computed by converting the fraction under the integral into a Taylor series
\begin{equation}
\label{eq:expansion_second_integral}
\begin{aligned}
	&\frac12\int_{\varphi_-(\varepsilon)}^{\varphi_+(\varepsilon)}\frac{(1-x_1\cos\varphi-y_1\sin\varphi)^3}{m^2 + \delta(1-x_1\cos\varphi-y_1\sin\varphi)^2}\,\d\varphi
	\\
	&\qquad=
	\frac12\int_{\varphi_-(\varepsilon)}^{\varphi_+(\varepsilon)}\frac{(1-x_1\cos\varphi-y_1\sin\varphi)^3}{M_1^2 + \delta(1-x_1\cos\varphi-y_1\sin\varphi)^2}\,\d\varphi+O(\varepsilon^2) 
	\\
	&\qquad=
	\frac{\sqrt2(1-x_0)^{1/2}(1-x_1)^3}{M_1^2+\delta(1-x_1)^2}+O(\varepsilon^2)
	=
	\frac{M_1 \sqrt{2} (1-x_0)^{7/2}}{1+\delta(1-x_0)^2}\varepsilon + O(\varepsilon^2)
	.
\end{aligned}
\end{equation}
In the last step, we used $1-x_1=M_1(1-x_0)$.
Moreover,
both integrals are odd analytic function of $\varepsilon$,
since $\varphi_+(-\varepsilon)=\varphi_-(\varepsilon)$
and the integrand in \eqref{eq:expansion_second_integral} is an even function w.r.t.\ $\varepsilon$.
Hence, the remainder terms
in \eqref{eq:original_resistance_series} and \eqref{eq:expansion_second_integral}
are actually $O(\varepsilon^3)$.

Second, we consider the triangles.
Using again $\varphi_+(-\varepsilon)=\varphi_-(\varepsilon)$,
we have $\nabla u_+(-\varepsilon)=\nabla u_-(\varepsilon)$
and $S_+(-\varepsilon)=-S_-(\varepsilon)$.
Hence, $\frac{S_+}{\abs{\nabla u_+}^2+\delta}+\frac{S_-}{\abs{\nabla u_-}^2+\delta}$
is an odd function of $\varepsilon$.
To expand $\nabla u_{\pm}=1/(1-x_0\cos\varphi_\pm-y_0\sin\varphi_\pm)$, we use
\begin{align}
&1-x_0\cos\varphi_\pm - y_0\sin\varphi_\pm
\nonumber
\\&\qquad= 
\frac{1-x_0\mp y_0\varepsilon\sqrt{2(1-x_0)-(1-x_0^2-y_0^2)\varepsilon^2}+(x_0^2+y_0^2-x_0)\varepsilon^2}
{1-2\varepsilon^2x_0 + \varepsilon^4(x_0^2+y_0^2)}
\nonumber
\\&\qquad= 
1 - x_0
\mp y_0 \sqrt{2(1-x_0)}
+
O(\varepsilon^2)
\label{eq:triangles_gradients_series}
\end{align}
and
\begin{equation}
	S_\pm
	=
	\frac12 (1 - M_1)(1 - x_0) \sqrt{2(1-x_0)} \, \varepsilon
	+
	O(\varepsilon^2)
	.
\label{eq:trangles_area_series}
\end{equation}
Thus,
\[
	\frac{S_{\pm}}{\abs{\nabla u_\pm}^2+\delta} =\frac{(1-M_1) (1-x_0)^{7/2}}{\sqrt{2}(\delta  (1 - x_0)^2+1)}\varepsilon + O(\varepsilon^2),
\]
\noindent
and
\begin{equation}
	\frac{S_+}{\abs{\nabla u_+}^2+\delta}+\frac{S_-}{\abs{\nabla u_-}^2+\delta}
	=
	\frac{\sqrt{2}(1-M_1) (1-x_0)^{7/2}}{\delta  (1 - x_0)^2+1}\varepsilon + O(\varepsilon^3).
\label{eq:triangles_resistance_series}
\end{equation}
By combining \eqref{eq:original_resistance_series}, \eqref{eq:expansion_second_integral} and \eqref{eq:triangles_resistance_series},
we have
\[
	\mathfrak{R}(\varepsilon) - \mathfrak{R}(0)
	=
	O(\varepsilon^3).
\]

Hence, a first-order Taylor expansion of $\mathfrak{R}$
does not yield enough information
and we have to use a higher order
Taylor expansion.
As we mentioned, 
$\mathfrak{R}(\varepsilon)$ is odd analytic in $\varepsilon$.
Thus, also the second-order term vanishes
and the third-order term
can be computed in a similar way by
expanding
\eqref{eq:original_resistance_series}--\eqref{eq:triangles_resistance_series} up to the $\varepsilon^3$ terms.
We arrive at
\begin{align*}
	&\frac{\mathfrak{R}(\varepsilon) - \mathfrak{R}(0)}{(1-x_0)^{5/2}}
	\\&\qquad=
	\frac{4(1-M_1)\sqrt{2}}{3 (1+\delta(1-x_0)^2)^3}\left[3y_0^2-(1-x_0)^2-\delta(1-x_0)^2((1-x_0)^2+y_0^2)\right]\varepsilon ^3
	\\&\qquad\qquad+O(\varepsilon^5).
\end{align*}

Thereby, since $\varepsilon>0$, we obtain that the sign of the variation of the resistance coincides with the sign of the expression
\[
	3y_0^2-(1-x_0)^2-\delta(1-x_0)^2((1-x_0)^2+y_0^2),
\]
in case that this expression is not zero.
It is interesting to note that the parameter $M_1$
does not appear.
Recall that we were assuming $x_2=1$ and $y_2=0$.
For an arbitrary case,
we must rotate the body in such a way
that the point $(x_2,y_2)$ will coincide with $(1,0)$. 

\begin{thm}
	\label{thm:nonoptimality2}
	Let $u\in \hat C$ and $\delta\ge 0$.
	Suppose that $u$ contains a conical part made up by
	the convex hull of
	a vertex $(x_0,y_0,-z_0)$
	($z_0>0$ and $(x_0,y_0)=(r_0\cos\varphi_0,r_0\sin\varphi_0)$ with $0\le r_0<1$)
	and an arc $(\cos\varphi,\sin\varphi,0) \in \partial\Omega$
	for $\varphi\in[\alpha,\beta]$ with $\alpha < \beta$.
	If there exists $\varphi\in[\alpha, \beta]$ such that
	\begin{equation}
		\frac{3r_0^2\sin^2(\varphi-\varphi_0)-\bracks{1-r_0\cos(\varphi-\varphi_0)}^2}
		{\bracks{1-r_0\cos(\varphi-\varphi_0)}^2\bracks{1+r_0^2-2r_0\cos(\varphi-\varphi_0)}}
		<\delta z_0^{-2}.
	\label{eq:criterion_2}
	\end{equation}
	
	\noindent Then $u$ is not optimal for $J_M$
	with $M=\delta^{-1/2}$ for $\delta>0$ and $M=\infty$ for $\delta=0$

\end{thm}

\begin{proof}
	The left-hand side of \eqref{eq:criterion_2}
	is continuous w.r.t.\ $\varphi$.
	Thus,
	w.l.o.g., we suppose $\varphi\in(\alpha,\beta)$.
	In order to apply the above arguments,
	we rotate and rescale the function $u$ via
	$\tilde u(x,y)=u(x\cos\varphi+y\sin\varphi,-x\sin\varphi +y\cos\varphi)/z_0$.
	Then, $J_M(u) = z_0^{-2}J_{z_0M}(\tilde u)$.
	The function $\tilde u$ contains a conical part
	made up by the apex $(X_0,Y_0,-1)$ with
	\begin{align*}
		X_0
		&=
		\phantom{-{}}x_0\cos\varphi + y_0\sin\varphi=r_0\cos(\varphi-\varphi_0);
		\\
		Y_0
		&=
		-x_0\sin\varphi + y_0\cos\varphi=r_0\sin(\varphi-\varphi_0);
	\end{align*}
	and an arc on the unit circle
	$\partial\Omega \times \{0\}$ containing point $(1,0,0)$ in its interior.
	Hence, applying the variation described in the beginning of the present section to $\tilde u$,
	we obtain
	that the change of the cost functional $J_{z_0M}(\tilde u)$ has the same sign as
	\[
		3Y_0^2-(1-X_0)^2-\tilde\delta(1-X_0)^2((1-X_0)^2+Y_0^2)
	\]
	where $\tilde\delta=(Mz_0)^{-2} = \delta \, z_0^{-2}$.
	Due to \eqref{eq:criterion_2},
	the variation has a negative sign.
	Hence, $\tilde u$ is not optimal for $J_{z_0 \, M}$
	and, consequently, $u$ is not optimal for $J_M$.
\end{proof}

We denote the left-hand side of inequality~\eqref{eq:criterion_2}
by $Z_0(r_0, \Delta\varphi)$, where $\Delta\varphi=\varphi-\varphi_0$.
In case $\delta = 0$,
non-optimality occurs if $Z_0(r_0, \Delta\varphi) < 0$.
In case $\delta > 0$,
the condition \eqref{eq:criterion_2} is equivalent to
\begin{equation*}
	Mz_0 <
	\begin{cases}
		\infty & \text{if } Z_0(r_0, \Delta\varphi) \le 0 \\
		Z_0(r_0, \Delta\varphi)^{-1/2} & \text{if } Z_0(r_0, \Delta\varphi) > 0.
	\end{cases}
\end{equation*}
In \cref{fig:visualize},
we plotted some level sets of
$Z_0(r_0, \Delta\varphi)^{-1/2}$
and
the level set 
$Z_0(r_0, \Delta\varphi) = 0$
(labeled with $\infty$)
in the polar coordinates $(r_0, \Delta\varphi)$.
\begin{figure}[ht]
	\centering
	\begin{tikzpicture}[scale=5]
		\tikzset{mylabel/.style  args={at #1 with #2}{
				postaction={decorate,
					decoration={
						markings,
						mark= at position #1
						with { \node [fill=white,draw=black] {#2};};
		} } } }
		\draw (0,0) circle (1);
		\begin{scope}
			\clip (0,0) circle (1);
			\def\z{0.25};
			\draw [thick,  domain=.6:1, samples=41, mylabel=at 0.4 with {$\z$}]
			plot ({\x}, {(1-\x)*sqrt((1/(\z*\z) *(\x-1)^2+1)/(3-1/(\z*\z)*(\x-1)^2))} );
			\draw [thick,  domain=.6:1, samples=41, mylabel=at 0.4 with {$\z$}]
			plot ({\x}, {-(1-\x)*sqrt((1/(\z*\z) *(\x-1)^2+1)/(3-1/(\z*\z)*(\x-1)^2))} );
			\def\z{0.5};
			\draw [thick,  domain=.3:1, samples=41, mylabel=at 0.4 with {$\z$}]
			plot ({\x}, {(1-\x)*sqrt((1/(\z*\z) *(\x-1)^2+1)/(3-1/(\z*\z)*(\x-1)^2))} );
			\draw [thick,  domain=.3:1, samples=41, mylabel=at 0.4 with {$\z$}]
			plot ({\x}, {-(1-\x)*sqrt((1/(\z*\z) *(\x-1)^2+1)/(3-1/(\z*\z)*(\x-1)^2))} );
			\def\z{1.0};
			\draw [thick,  domain=.0:1, samples=41, mylabel=at 0.3 with {$\z$}]
			plot ({\x}, {(1-\x)*sqrt((1/(\z*\z) *(\x-1)^2+1)/(3-1/(\z*\z)*(\x-1)^2))} );
			\draw [thick,  domain=.0:1, samples=41, mylabel=at 0.3 with {$\z$}]
			plot ({\x}, {-(1-\x)*sqrt((1/(\z*\z) *(\x-1)^2+1)/(3-1/(\z*\z)*(\x-1)^2))} );
			\draw [thick,  domain=-1.0:1, samples=41, mylabel=at 0.5 with {$\infty$}] plot ({\x}, {(1-\x)*sqrt((1)/(3))} ) ;
			\draw [thick,  domain=-1.0:1, samples=41, mylabel=at 0.5 with {$\infty$}] plot ({\x}, {-(1-\x)*sqrt((1)/(3))} ) ;
		\end{scope}
	\end{tikzpicture}
	\caption{
		This shows some level sets
		of the function $Z_0(r_0, \Delta\varphi)^{-1/2}$, see \eqref{eq:criterion_2},
		in
		the polar coordinate system $(r_0, \Delta\varphi)$.
	}
	\label{fig:visualize}
\end{figure}
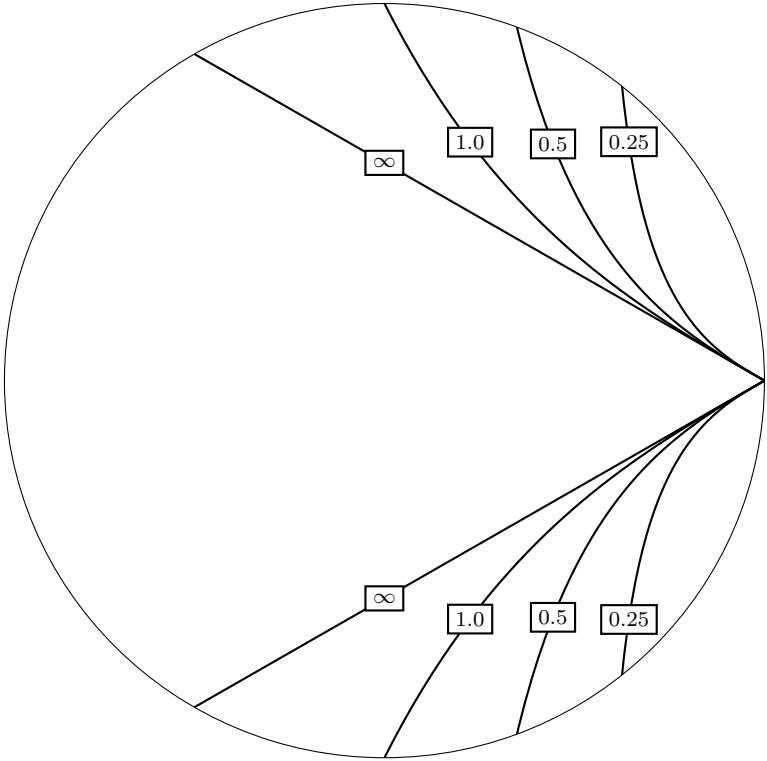

\cref{thm:nonoptimality2} applies to the rescaled version of Newton's problem.
For later reference, we also give a formulation which can be directly applied to
the original problem \eqref{problem:J}.
\begin{corollary}
	\label{cor:nonopt}
	Let $M \ge 0$ and $u\in C_M$ be given.
	Suppose that $u$ contains a conical part made up by
	the convex hull of
	a vertex $(x_0,y_0,-z_0)$
	($z_0>0$ and $(x_0,y_0)=(r_0\cos\varphi_0,r_0\sin\varphi_0)$ with $0<r_0<1$)
	and a boundary arc $(\cos\varphi,\sin\varphi,0) \in \partial\Omega$
	for $\varphi\in[\alpha,\beta]$ with $\alpha < \beta$.
	If there exists $\varphi\in[\alpha, \beta]$ such that
	\begin{equation}
		\frac{3r_0^2\sin^2(\varphi-\varphi_0)-\bracks{1-r_0\cos(\varphi-\varphi_0)}^2}
		{\bracks{1-r_0\cos(\varphi-\varphi_0)}^2\bracks{1+r_0^2-2r_0\cos(\varphi-\varphi_0)}}
		<z_0^{-2}.
	\label{eq:criterion_3}
	\end{equation}
	\noindent Then $u$ is not optimal for $J$.
\end{corollary}

\section{Non-optimality in the Class of all Convex Function of Suggested Solutions in the Literature}
\label{sec:non-optimal_suggestions}
In this section,
we apply \cref{cor:nonopt}
to some conjectured solutions.
In particular, we address the contributions
\cite{LachandRobertPeletier2001,Wachsmuth,LokutsievskiyZelikin2}.

\subsection{Conjectured solutions by \citeauthor{LachandRobertPeletier2001} (\citeyear{LachandRobertPeletier2001})}
We proceed in chronological order
and start with the bodies given in
\cite{LachandRobertPeletier2001}.
Therein, the authors studied Newton's problem
in a restricted class of functions
and obtained bodies
which are the convex hull of
$\Omega \times \{0\} \cup N_0 \times \{-M\}$,
where $N_0 \subset \R^2$ is a regular polygon centered at $0$.
We note that the (global) non-optimality of these bodies
was already observed in
\cite{LachandNumeric,Wachsmuth}
via the comparison with the numerical solutions.
We will check that the (local) non-optimality also follows from \cref{cor:nonopt}.
Let us assume that $N_0$ is a regular polygon with $k \ge 2$
vertices
and we rotate $N_0$ such that one vertex is given by $(x_0, 0)$
for some $x_0 \in (0,1)$.
Then, it is clear that the body
contains a conical part with parameters
\begin{align*}
	(x_0, y_0, z_0) &= (x_0, 0, -M) = (r_0\cos\varphi_0, r_0\sin\varphi_0, -M)
	&
	r_0 &= x_0,
	\\
	\varphi_0 &= 0,
	&
	\beta &= -\alpha = \pi / k.
\end{align*}
The body with $M = 1.0$ is shown in
\cref{fig:conical_parts} (top left).
\begin{figure}[tp]
	\centering
	\includegraphics[width=.48\textwidth]{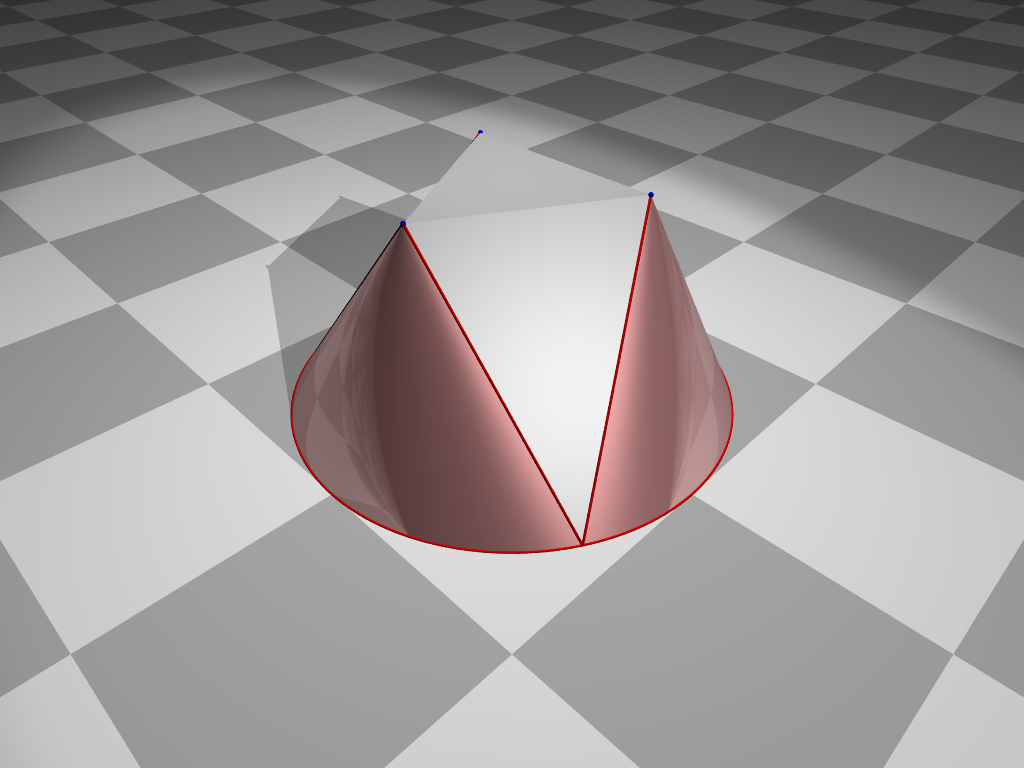}%
	\hspace{.2cm}%
	\includegraphics[width=.48\textwidth]{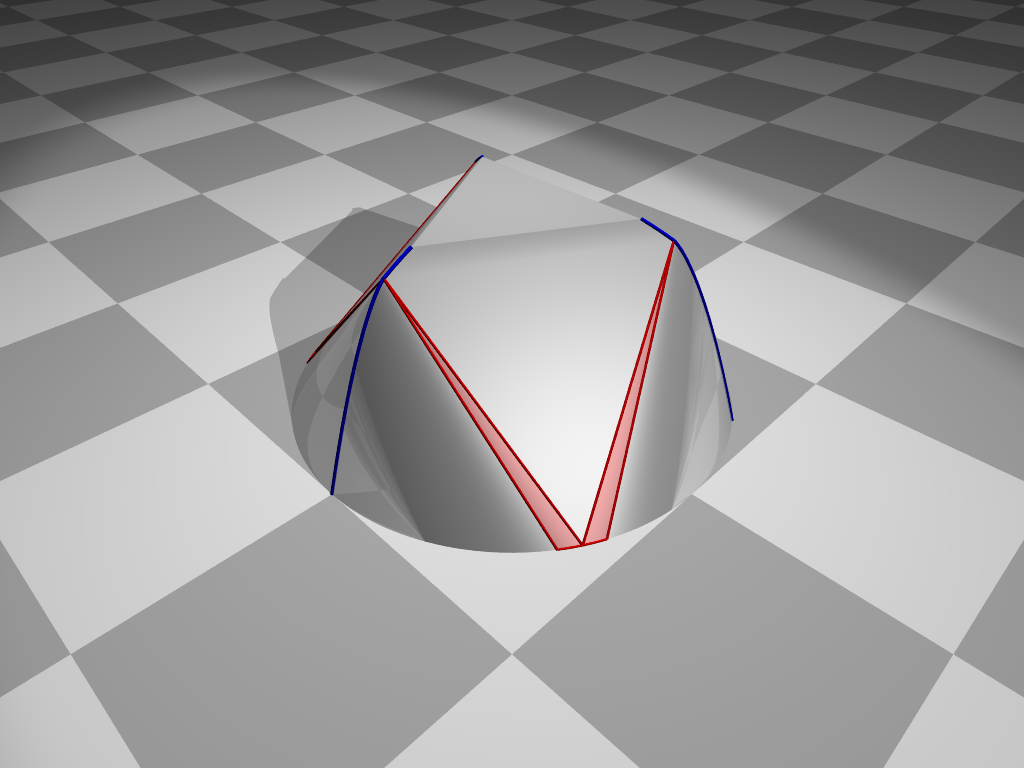}\\[.2cm]
	\includegraphics[width=.48\textwidth]{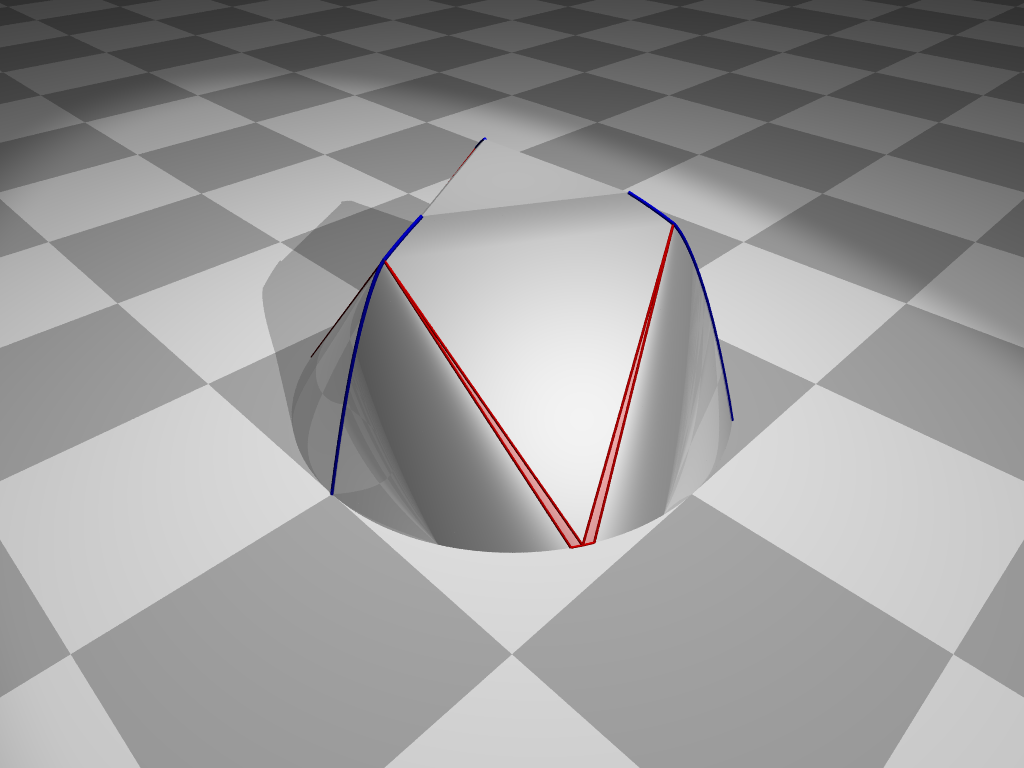}%
	\hspace{.2cm}%
	\includegraphics[width=.48\textwidth]{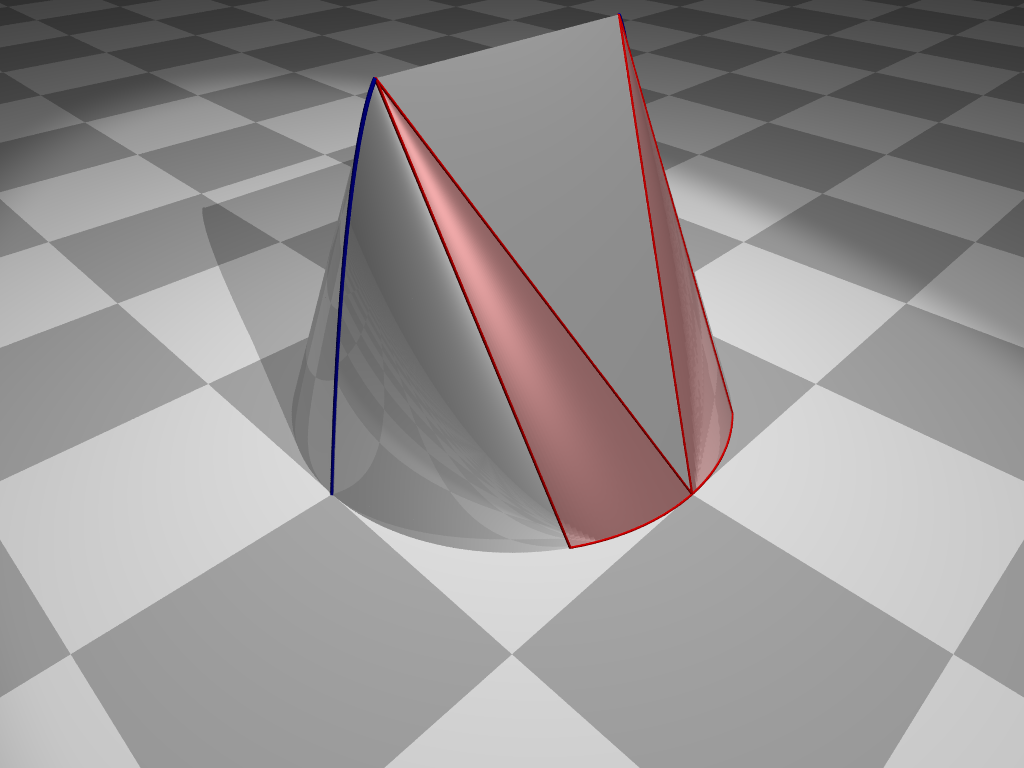}\\[.2cm]
	\includegraphics[width=.48\textwidth]{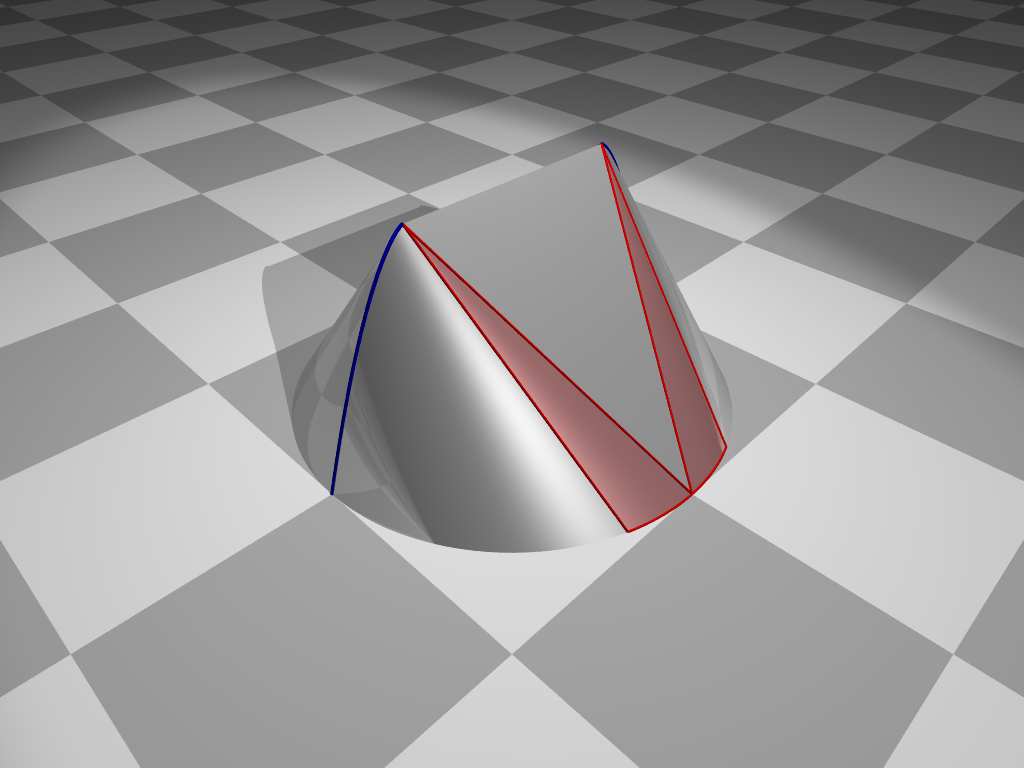}%
	\hspace{.2cm}%
	\includegraphics[width=.48\textwidth]{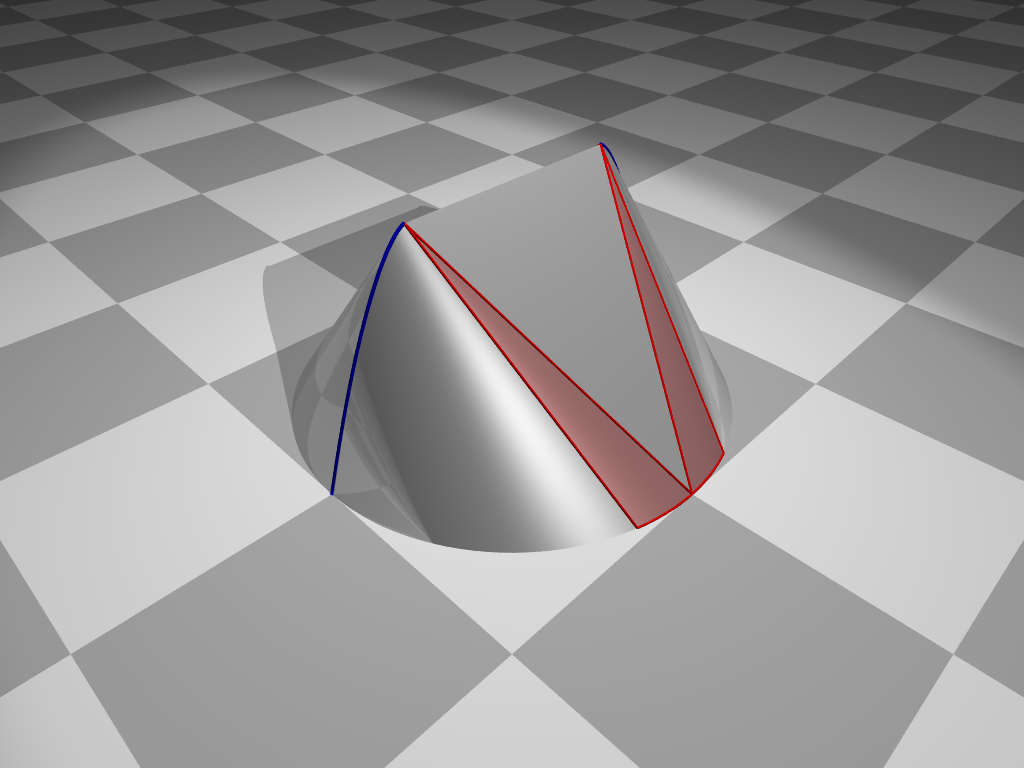}%
	\caption{%
		Illustration of conjectured solutions (shown upside down) containing non-optimal conical parts
		(highlighted red):
		\cite{LachandRobertPeletier2001} with $M = 1.0$ (top left),
		\cite{Wachsmuth} with $(M, k) = (0.9, 3)$ (top right)
		and with $(M, k) = (1.0, 3)$ (middle left),
		\cite{LokutsievskiyZelikin2} with $M = 1.5$ (middle right),
		$M = 5.0$ (bottom left),
		$M = \infty$ (bottom right).
		The bodies in the bottom row are rescaled to height $1.0$, cf.\ \eqref{problem:J_M}.
		All the bodies are constructed as the convex hull of the blue points
		and the base $\Omega \times \{0\}$.
	}
	\label{fig:conical_parts}
\end{figure}
Now, it is easy to check that the left-hand side of inequality
\eqref{eq:criterion_3}
is negative
for $\varphi = 0 \in (\alpha,\beta)$
and, therefore,
\eqref{eq:criterion_3} holds true.

\textit{Hence, the bodies 
suggested by \citeauthor{LachandRobertPeletier2001} (\citeyear{LachandRobertPeletier2001})
cannot be optimal for any value of $M > 0$
and $k \ge 2$.}

\subsection{Conjectured solutions by \citeauthor{Wachsmuth} (\citeyear{Wachsmuth})}
\label{subsec:conj_wachsmuth}
Next, we investigate the structural conjecture
from \cite[Section~3]{Wachsmuth}
where the author has supposed that
optimal bodies for a height $M \in (0,\bar M)$,
with $\bar M  \in (1.4, 1.5)$,
have the following structure.
There exists $k\in\N$, $k\ge 3$, and 
a convex function $g \colon [0,1] \to [-M, 0]$,
$g(0) = -M$, $g(1) = 0$,
such that the optimal body is the convex hull
of the set
\begin{equation*}
	\partial\Omega \times \{0\}
	\cup
	\{
		(r \, \cos(2\,i\,\varphi), r \, \sin(2\,i\,\varphi), g(r)), \; i=0,\ldots,k-1,\; r \in [0,1]
	\}
	,
\end{equation*}
where $\varphi = \pi / k$.
Examples with $M = 1.0$ and $M = 0.9$
are depicted in \cref{fig:conical_parts}
(top right and middle left).
The extremal lines
\begin{equation*}
	\{
		(r \, \cos(2\,i\,\varphi), r \, \sin(2\,i\,\varphi), g(r)), \; i=0,\ldots,k-1,\; r \in [0,1]
	\}
\end{equation*}
are highlighted in blue.
Under some natural assumptions on $g$,
the problem becomes a one-dimensional problem
of calculus of variations and can be solved for $g$ by the corresponding
Euler-Lagrange equations.
The obtained class of solutions
contains a conical part (denoted by Region~II therein),
see \cite[Figure~6]{Wachsmuth}
and \cref{fig:conical_parts} (middle left).
However, for the results presented in
\cite[Table~3]{Wachsmuth}
(reproduced and extended in \cref{tab:conjectured_solutions}),
only the
solution corresponding to $M = 1.0$ (the height parameter is denoted by $L$ in \cite{Wachsmuth})
and symmetry parameter $k = 3$ (denoted by $m$ in \cite{Wachsmuth})
contains this conical part
and
for all other presented solutions, this conical part vanishes.
Moreover, it can be checked
that \cref{cor:nonopt}
applies to this solution with $M = 1.0$ and $k = 3$.
Hence, the structural conjecture of
\cite{Wachsmuth}
cannot be true for this height $M = 1.0$.

The (non-optimal) body from \cite[Section~3]{Wachsmuth} with $(M,k) = (1.0, 3)$ is displayed in \cref{fig:conical_parts} (middle left).
Note that the conical parts are rather small.
We believe that the non-optimality of this body
is (informally speaking)
only due to these small conical parts.
Therefore, we expect that the objective value
can be improved only by a small amount
and this seems to be hard to achieve via numerical methods.
We also show the body corresponding to 
$(M,k) = (0.9, 3)$ (top right) which has a larger conical part.

In \cref{tab:conjectured_solutions},
we present an updated\footnote{%
	There is one significant difference to
	\cite[Table~3]{Wachsmuth}:
	the given objective value corresponding to $M = 1.0$ and $k = 3$
	was suboptimal and has been corrected in \cref{tab:conjectured_solutions}.%
} and completed version of
\cite[Table~3]{Wachsmuth}.
\begin{table}
	\centering
	\scriptsize
	\setlength{\tabcolsep}{2.5pt} 
	\begin{tabular}{l*{6}{l}}
		\toprule
		$M$ \textbackslash\ \rlap{$k$} & \multicolumn{1}{c}{3} & \multicolumn{1}{c}{4} & \multicolumn{1}{c}{5} & \multicolumn{1}{c}{6} & \multicolumn{1}{c}{7} & \multicolumn{1}{c}{8} \\
		\midrule
		$1.5$ & $0.6999489$    & $0.7123013$    & $0.7218282$    & $0.7285828$    & $0.7334158$    & $0.7369567$   \\
		$1.4$ & \underline{$0.7677364$}    & $0.7792767$    & $0.7887148$    & $0.7955246$    & $0.8004380$    & $0.8040548$   \\
		$1.3$ & \underline{$0.8441426$}    & $0.8544163$    & $0.8635882$    & $0.8703656$    & $0.8753090$    & $0.8789708$   \\
		$1.2$ & \underline{$0.9303614$}    & $0.9387924$    & $0.9474690$    & $0.9540986$    & $0.9590059$    & $0.9626708$   \\
		$1.1$ & \underline{$1.0277294$}    & $1.0335975$    & $1.0414890$    & $1.0478237$    & $1.0526095$    & $1.0562233$   \\
		$1.0$ & \underline{$1.1377294$ CN} & $1.1401510$    & $1.1468980$    & $1.1527533$    & $1.1573095$    & $1.1608036$   \\
		$0.9$ & $1.2619895$ CN & \underline{$1.2599052$}    & $1.2650665$    & $1.2702151$    & $1.2744081$    & $1.2776964$   \\
		$0.8$ & $1.4022724$ CN & \underline{$1.3944389$}    & $1.3974884$    & $1.4016551$    & $1.4053219$    & $1.4082995$   \\
		$0.7$ & $1.5604532$ CN & \underline{$1.5454605$}    & $1.5457807$    & $1.5486392$    & $1.5515845$    & $1.5541263$   \\
		$0.6$ & $1.7392117$ CN & $1.7147979$    & \underline{$1.7116824$}    & $1.7128535$    & $1.7148496$    & $1.7168060$   \\
		$0.5$ & $1.9438498$ CN & $1.9043842$    & $1.8970564$    & \underline{$1.8961062$}    & $1.8968909$    & $1.8980899$   \\
		$0.4$ & $2.1775037$ CN & $2.1176695$ C  & $2.1038722$    & $2.1003292$    & \underline{$2.0996063$}    & $2.0998537$   \\
		$0.3$ & $2.4390875$ CN & $2.3611648$ CN & $2.3354683$ C  & $2.3276201$ C  & $2.3250270$    & $2.3241056$   \\
		$0.2$ & $2.7171218$ CN & $2.6385156$ CN & $2.5993030$ C  & $2.5828003$ C  & $2.5760193$ C  & $2.5731257$ C \\
		$0.1$ & $2.9786691$ CN & $2.9338815$ CN & $2.8990387$ CN & $2.8754404$ CN & $2.8612811$ C  & $2.8533036$ C \\
		\bottomrule
	\end{tabular}
	\caption{%
		Conjectured optimal values using the conjecture from \cite[Section~3]{Wachsmuth}.
		A ``C'' indicates that this solution contains a conical part.
		For the solutions marked by ``N'', \cref{cor:nonopt} applies and, therefore, those solutions cannot be optimal
		in the class of all convex bodies $C_M$.
		The minimal (conjectured) solutions for each $M$ are underlined.
	}
	\label{tab:conjectured_solutions}
\end{table}
In this table,
we underlined the best solution in each row, i.e.\ 
for each height parameter $M$.
Note that for $M = 1.5$ a better solution 
was obtained numerically in \cite[Section~2]{Wachsmuth}
whereas for $M \le 0.3$
a structured solution with $k = 9$ produces better values than
the solutions given in the table.
Hence, we do not underline solutions
in the lines corresponding to $M = 1.5$
and $M \le 0.3$.

For each solution presented in \cref{tab:conjectured_solutions},
we checked whether this solution contains a conical part
(indicated by ``C'')
and whether \cref{cor:nonopt}
applies to this conical part
and provides the non-optimality
(indicated by ``N'').
For each fixed $k$
it seems that conical parts appear for small values of
$M$ (depending on $k$)
and that, eventually, this conical part becomes non-optimal.
However, for $k \ge 4$
the non-optimality appears only for ``very small''
values of $M$ and for these values, $k + 1$
provides a better solution.
Hence, for $M \le 0.9$ (and, therefore, $k \ge 4$)
we cannot apply \cref{cor:nonopt}
and we cannot disprove the conjecture of \cite[Section~3]{Wachsmuth}.
For $M$ between $1.0$ and $1.4$ the situation is different.
Here, the best results (according to the structural conjecture of \cite[Section~3]{Wachsmuth}) are obtained by $k = 3$
and these contain non-optimal conical parts for
heights $M$ that are smaller than approximately $1.0$.
In particular, we can apply \cref{cor:nonopt}
for the height $M = 1.0$
and therefore,
the conjecture of \cite[Section~3]{Wachsmuth}
is disproved for this value.
For $M$ bigger than $1.1$, the solutions with $k = 3$
do not contain conical parts and therefore,
we cannot disprove the conjecture
for $M$ between $1.1$ and $1.4$.

In \cref{tab:conjectured_solutions_detail},
we list some more values for $M \in [0.9, 1.1]$
and $k \in \{3,4\}$.
\begin{table}
	\centering
	\begin{tabular}{l*{3}{l}}
		\toprule
		$M$ \textbackslash\ $k$ & \multicolumn{1}{c}{3} & \multicolumn{1}{c}{4} \\
		\midrule
		$1.10$ &  \underline{$1.0277294$   } & $1.0335975$ \\
		$1.09$ &  \underline{$1.0381352$ C } & $1.0437014$ \\
		$1.08$ &  \underline{$1.0486688$ CN} & $1.0539240$ \\
		$1.07$ &  \underline{$1.0593317$ CN} & $1.0642667$ \\
		$1.06$ &  \underline{$1.0701256$ CN} & $1.0747311$ \\
		$1.05$ &  \underline{$1.0810520$ CN} & $1.0853184$ \\
		$1.04$ &  \underline{$1.0921125$ CN} & $1.0960303$ \\
		$1.03$ &  \underline{$1.1033089$ CN} & $1.1068681$ \\
		$1.02$ &  \underline{$1.1146427$ CN} & $1.1178333$ \\
		$1.01$ &  \underline{$1.1261157$ CN} & $1.1289274$ \\
		$1.00$ &  \underline{$1.1377294$ CN} & $1.1401510$ \\
		\bottomrule
	\end{tabular}%
	\hspace{1cm}%
	\begin{tabular}{l*{3}{l}}
		\toprule
		$M$ \textbackslash\ $k$ & \multicolumn{1}{c}{3} & \multicolumn{1}{c}{4} \\
		\midrule
		$1.00$ &  \underline{$1.1377294$ CN} & $1.1401510$ \\
		$0.99$ &  \underline{$1.1494856$ CN} & $1.1515072$ \\
		$0.98$ &  \underline{$1.1613861$ CN} & $1.1629969$ \\
		$0.97$ &  \underline{$1.1734324$ CN} & $1.1746215$ \\
		$0.96$ &  \underline{$1.1856264$ CN} & $1.1863837$ \\
		$0.95$ &  \underline{$1.1979697$ CN} & $1.1982830$ \\
		$0.94$ &  $1.2104642$ CN & \underline{$1.2103218$} \\
		$0.93$ &  $1.2231117$ CN & \underline{$1.2225019$} \\
		$0.92$ &  $1.2359138$ CN & \underline{$1.2348243$} \\
		$0.91$ &  $1.2488725$ CN & \underline{$1.2472919$} \\
		$0.90$ &  $1.2619895$ CN & \underline{$1.2599052$} \\
		\bottomrule
	\end{tabular}
	\caption{%
		Similar as \cref{tab:conjectured_solutions}, but for different values of $M$ and $k$.
	}
	\label{tab:conjectured_solutions_detail}
\end{table}
This table suggests the following observations:
\begin{itemize}
	\item
		For $M \ge 1.09$, the bodies conjectured in \cite[Section~3]{Wachsmuth} might be optimal
		since these bodies do not contain conical parts 
		or their conical parts do not satisfy~\cref{cor:nonopt}.
	\item
		For $M \in [0.95, 1.08]$, the conjectured optimal bodies contain
		a non-optimal conical part and, therefore, the structural conjecture
		of \cite[Section~3]{Wachsmuth} is disproved
		for these values of $M$.
	\item
		For $M \in [0.90, 0.94]$, 
		our non-optimality result does not apply
		to
		the conjectured bodies with $k = 4$
		and these bodies possess better values than those with $k = 3$.
		However,
		the bodies with symmetry parameter $k = 3$
		are not locally optimal in $C_M$
		by~\cref{cor:nonopt}.
		It is also clear that our variation from \cref{sec:nonoptimality}
		can be modified to produce bodies with a threefold symmetry
		which possess smaller objective values than those indicated in \cref{tab:conjectured_solutions_detail}
		for $k = 3$.
		In particular, these improved values could be smaller than the
		corresponding values with $k = 4$ from \cref{tab:conjectured_solutions_detail}
		and
		this would disprove the structural conjecture of \cite{Wachsmuth} for some values
		of $M$ around $0.94$.
		This is subject to future research.
\end{itemize}

\textit{To summarize,
\cref{cor:nonopt}
disproves the conjectured bodies from \cite[Section~3]{Wachsmuth}
at least for $M \in [0.95, 1.08]$, see \cref{tab:conjectured_solutions_detail}.
It does not apply for $M \le 0.94$ and $M \in [1.09, 1.4]$,
and for these values, the conjecture might be true.}

\subsection{Conjectured solutions by \citeauthor{LokutsievskiyZelikin2} (\citeyear{LokutsievskiyZelikin2})}
\label{subsec:lokutsievskiy_zelikin_solution}
In \cite{LokutsievskiyZelikin2},
the authors study the class $E_M$ of convex bodies
of height $M$
which can be written as the convex hull
of the union of the base $\Omega \times \{0\}$
and of a convex curve $z=v^*(x_1)$
in the plane $\{x_2 = 0\}$
(we keep notations of \cite{LokutsievskiyZelikin2},
and $v^*$ denotes the Legendre-Young-Fenchel
transform of a convex function $v$).
We note that this approach is similar to \cref{subsec:conj_wachsmuth}
with $k = 2$.
The authors proved local optimality
of such bodies in the corresponding class
(see \cite[Theorem~9.1]{LokutsievskiyZelikin2}).

In this paper, there is a table
with numerically found parameters
of the locally optimal curve $v^*$
for some different values of the height $M$
(see \cite[Table~1]{LokutsievskiyZelikin2}).
The solution $v^*$ has 
a horizontal line segment
in the front of the body, 
since $v$
has a corner at $0$.
It can be checked that
all the bodies from
the table
contain a conical part.
Exemplarily, we have shown the bodies corresponding to $M = 1.5$
in \cref{fig:conical_parts} (middle right)
and to $M = 5.0$ (rescaled to height $1.0$, bottom left).
The conical part is given by
the vertex
$(x_0, y_0, z_0) = (v'(+0), 0, -M)$
and
the arc with angles $[\alpha,\beta] = [\pi/2 - \varepsilon, \pi/2]$
for some\footnote{Using the notation from \cite{LokutsievskiyZelikin2}, $\varepsilon=\arcsin [r(p_0)/(M+v'(0)r(p_0))]$, e.g.\ $\varepsilon \approx 0.574610$ for $M=1.5$ and $\varepsilon \approx 0.330507 $ for $M=5.0$.}
 $\varepsilon > 0$.
Now, it is easily checked that inequality \eqref{eq:criterion_3}
from \cref{cor:nonopt}
is fulfilled for
$r_0 = v'(+0)$, $\varphi_0 = 0$, $\varphi = \beta = \pi/2$ and $z_0 = M$.
Therefore,
this conical part is always non-optimal in the class $C_M$ of all convex bodies.
Hence,
it seems that the optimal bodies in the class $E_M$ are never optimal in $C_M$.

A similar approach can be used for the limiting problem
of minimizing $J_\infty$ in the class $E_1$.
Using the same strategy, one obtains the values
(with the notation of \cite{LokutsievskiyZelikin2})
\begin{align*}
	p_0       &\approx 3.167203701258, &
	r(p_0)    &\approx 0.3451623687826, \\
	v'(+0)    &\approx 0.5300674211893, &
	J_\infty(u) &\approx 2.140225047120 
	.
\end{align*}
The corresponding body is shown in \cref{fig:conical_parts} (bottom right).
Again, this body has a conical part (with $\varepsilon\approx 0.296085$),
which is non-optimal by \cref{cor:nonopt}.

\medskip

\textit{Thus, the minimizers of $J_\infty$ in $\hat C$ do not belong to~$E_1$. 
Hence, the family found in \cite{LokutsievskiyZelikin2} 
cannot be asymptotically optimal for $J$ in $C_M$ by~\cref{prop:asymp_limit} 
despite the fact that it is locally optimal 
in $E_M$ by \cite[Theorem~9.1]{LokutsievskiyZelikin2}.}

\subsection*{Acknowledgement}
Gerd Wachsmuth acknowledges fruitful discussions
with Luca Landwehrjohann back in 2017
which spawned the seeds for some of the ideas in the present paper.

The work of Lev Lokutsievskiy was performed at the Steklov International Mathematical Center and supported by the Ministry of Science and Higher Education of the Russian Federation (agreement no. 075-15-2019-1614).
The work of Gerd Wachsmuth was partially supported by the DFG Grant \emph{Approximation of Non-Smooth Optimal Convex Shapes with Applications in Optimal Insulation and Minimal Resistance}
(Grant No.\ WA 3636/5-2)
within the Priority Program SPP 1962 (Non-smooth
and Complementarity-based Distributed Parameter Systems: Simulation and Hierarchical Optimization).
The work of Mikhail Zelikin was supported by Russian Foundation for Basic Research under grant 20-01-00469.

\printbibliography

\end{document}